\documentclass[12pt,reqno]{amsart}
\usepackage{bm}
\usepackage{graphicx}
\usepackage{adjustbox}
\usepackage{amsthm,amssymb,amsmath}
\usepackage[T1]{fontenc}
\usepackage[utf8]{inputenc}
\usepackage[colorlinks=true, linkcolor=blue, urlcolor=blue, citecolor=blue, anchorcolor=blue]{hyperref}
\usepackage{xcolor}
\usepackage{mathrsfs}
\usepackage{enumitem}
\usepackage{setspace}
\usepackage{yfonts}
\usepackage{float}
\usepackage{mathtools}
\usepackage[all,cmtip]{xy}
\usepackage{todonotes}
\usepackage{tikz-cd}
\usepackage{relsize}

\newtheorem{manualtheoreminner}{Theorem}
\newenvironment{manualtheorem}[1]{%
\renewcommand\themanualtheoreminner{#1}%
  \manualtheoreminner
}{\endmanualtheoreminner}

\usepackage[backend=bibtex]{biblatex}

\tikzcdset{row sep/normal=50pt, column sep/normal=50pt}

\newtheorem{lemma}{Lemma}[section]
\newtheorem{theorem}[lemma]{Theorem}

\theoremstyle{definition}

\newtheorem{supplement}[lemma]{Supplement}

\numberwithin{equation}{section}

\begin{document}
\title[Multiprojective bundles and Symmetric powers]{Isomorphisms and automorphisms of multiprojective bundles and symmetric powers of projective bundles}
\author[A. Bansal]{Ashima Bansal}
\address{Indian Institute of Science Education and Research Tirupati, Srinivasapuram, Yerpedu Mandal, Tirupati Dist, Andhra Pradesh, India – 517619.}
\email{ashimabansal@students.iisertirupati.ac.in}
\author[S. Sarkar]{Supravat Sarkar}
\address{\scshape Fine Hall, Princeton, NJ 700108}
\email{ss6663@princeton.edu}
\author[S. Vats]{Shivam Vats}
\address{}
\email{shivamvatsaaa@gmail.com}

\begin{abstract}
We describe when two multiprojective bundles (fibre products of projective bundles over the same base) over projective spaces are isomorphic as abstract varieties. We also describe when two relative symmetric powers of projective bundles over projective spaces are isomorphic. Finally, we describe the automorphisms of multiprojective bundles and relative symmetric powers of projective bundles over projective spaces.
\end{abstract}
\maketitle
\section{Introduction}
\noindent Projective spaces and multiprojective spaces (products of projective spaces) are the simplest examples of complete algebraic varieties. Relativising the constructions of projective and multiprojective spaces, one gets the notion of projective bundles and multiprojective bundles, the latter being the fibre product of projective bundles over the same base. With these constructions in mind, one naturally leads to studying projective bundles and multiprojective bundles over projective spaces. A natural question in this area is when two multiprojective bundles over projective spaces are isomorphic. In \cite{Sa}, this question is answered for projective bundles. In {\cite[Theorem B]{BSV1}}, this question is answered for multiprojective bundles over different projective spaces. One of our goals in this article is to tackle the remaining case, that is, to characterize when two multiprojective bundles over the same projective space are isomorphic as abstract varieties.

Throughout, we work over the field $\mathbb{C}$ of complex numbers, although the proofs of Theorems \ref{A} and \ref{autmult} hold over any algebraically closed field. A \textit{variety} is an integral, finite type, separated scheme over $\mathbb{C}$. We follow the conventions of projectivizations of vector bundles as in {\cite[Chapter 2, Section 7]{Hart}}. A \textit{projective bundle} over a variety $X$ is the projectivization of a vector bundle $E$ over $X$, the projective bundle is called \textit{proper} if rank $E\,\geq\, 2$. For vector bundles $E_1, E_2,\cdots, E_r$ on a variety $X$, we denote the fibre product $ \mathbb{P}_X(E_{1}) \times_{X} \mathbb{P}_X(E_{2}) \times_{X}\cdots\times_{X }\mathbb{P}_X(E_{r})$ by $ \mathbb{P}_{X}(E_{1},\cdots,E_{r})$, and we call this a \textit{multiprojective bundle} over $X$. We will sometimes omit the subscript $X$ if it is clear from the context. For a variety $Y$, we denote the automorphism group of $Y$ by Aut$(Y)$. Throughout, for a positive integer $r$, $S_r$ denotes the symmetric group on $r$ letters.

Suppose $n,r,s $ are positive integers, $\underline{E}=(E_1, E_2,\cdots, E_r)$ and $\underline{F}=(F_1, F_2,\cdots, F_s)$ be tuples of vector bundles over $\mathbb{P}^n$, each of rank $\,\geq\, 2$. If $\mathbb{P}_{\mathbb{P}^n}(\underline{E}) \,\cong\, \mathbb{P}_{\mathbb{P}^n}(\underline{F})$ as abstract varieties, then a comparison of Picard ranks shows that $r=s$. Our result characterizes when $\mathbb{P}_{\mathbb{P}^n}(\underline{E})$ can be isomorphic to $\mathbb{P}_{\mathbb{P}^n}(\underline{F})$ as abstract varieties.

\begin{manualtheorem}{A}\label{A}
Let $n,r \,\geq\, 1$, $\underline{E}=(E_1, E_2,\cdots, E_r), \underline{F}=(F_1, F_2,\cdots, F_r)$ be $r$-tuples of vector bundles over $\mathbb{P}^n$, each of rank $\,\geq\, 2$. Suppose $\mathbb{P}_{\mathbb{P}^n}(\underline{E}) \,\cong\, \mathbb{P}_{\mathbb{P}^n}(\underline{F})$ as abstract varieties. Then there exist $\psi \,\in\, \textnormal{Aut}(\mathbb{P}^n), L_{i } \,\in\, \textnormal{Pic}(\mathbb{P}^n) $, and $ \sigma \,\in\, S_{r}$ such that $ E_{i} \,\cong\, \psi^{\ast}F_{\sigma(i)} \otimes L_{i}$ for all $i = 1,\cdots,r$.
\end{manualtheorem}
For a map of varieties $Y\to X$ and a positive integer $r$, let $S^r_X(Y)$ denote the $r$th relative symmetric power of $Y$ over $X$; that is, the quotient of the fibre product $ Y \times_{X} Y \times_{X}\cdots\times_{X }Y$ ($r$-times) by the action of the symmetric group $S_r$, which permutes the factors. While studying multiprojective bundles, a related concept is the notion of relative symmetric powers of projective bundles, which is, by definition, a quotient of a multiprojective bundle. A natural question is when two relative symmetric powers of projective bundles over projective spaces are isomorphic as abstract varieties. We answer this question in the following theorem. 
\begin{manualtheorem}{B}\label{sym}
Let $ p,q,r,s \,\geq\, 2,$ and $m,n \,\geq\, 1$
 be positive integers. Let $E$ be a vector bundle over $ \mathbb{P}^n$  of rank $p+1$, $F$ a vector bundle over $ \mathbb{P}^n$ of rank $ q+1$. Suppose there exists an isomorphism
 $$ S_{\mathbb{P}^n}^{r}\mathbb{P}(E)\,\xlongrightarrow{\phi}\, S_{\mathbb{P}^m}^s\mathbb{P}(F).$$ 
 Then $m\,=\,n,$ $r\,=\,s$, $ p\,=\,q$, and there exists $\psi \,\in\, \textnormal{Aut}(\mathbb{P}^n)$ such that $ \phi$ is over $ \psi$, and $ E \,\cong\, \psi^{\ast}F \otimes L$ for some line bundle $L$ over $\mathbb{P}^n.$
 \end{manualtheorem}
The automorphism group is a natural invariant of a variety. It gives information about the global symmetries of a variety. Automorphism groups of projective bundles have been studied by several authors  (see \cite{BFT}, \cite{Fa}, \cite{Fo}, \cite{Ko} and \cite{Ma}). However, automorphisms of multiprojective bundles and symmetric powers of projective bundles are relatively less studied. In this paper, we describe the automorphism groups of multiprojective bundles and symmetric powers of projective bundles over projective spaces. To state our results, we introduced a few notations. For a map of varieties $f\,:\,Y\,\to\, X$, let Aut$_X(Y)\,=\,\{\phi\in\text{ Aut}(Y)\,|\, f\,=\,f\circ\phi\}$, and $\textnormal{Aut}_{X\to X}(Y)\,=\,\{\phi\in\text{ Aut }(Y)\,|\,\text{ there is } \psi\in \text{Aut}(X) \text{ with } \psi\circ f\,=\,f\circ\phi\}.$  An automorphism $\phi$ of $Y$ is called \textit{over an automorphism of $X$} if $\phi\, \in\,$ $\textnormal{Aut}_{X\to X}(Y),$ and is called \textit{over $X$} if $\phi\, \in\,$ Aut$_X(Y)$. 

 Our results are:
 
\begin{manualtheorem}{C} \label{autmult}
Let $n, \ell ,a_{1},\cdots,a_{\ell}$ be positive integers, $ P_{1}, P_{2}, \cdots, P_{\ell}$ be proper projective bundles over $ \mathbb{P}^n$, and assume they are pairwise non-isomorphic over $ \mathbb{P}^n$. Also, assume that neither of the following conditions holds:
\begin{enumerate}
\item[$(1)$] All the bundles $P_{i}$'s are trivial.
\item[$(2)$] There is an $i$ with $ a_{i}\,=\,1$, $ P_{i} \,\cong\, \mathbb{P}_{\mathbb{P}^n}(T_{\mathbb{P}^n})$ over $ \mathbb{P}^n,$ and all other $P_{j}$ for all $ j \neq i$ are trivial.
\end{enumerate}
Let $ P\,=\, \prod_{i}P_{i}^{a_{i}},$ where $ \prod$ and powers are fiber products over $ \mathbb{P}^n.$  Let $H$ be the subgroup of all automorphisms $ \psi$ of $\mathbb{P}^n $ such that $ (\underbrace{P_{1},\cdots,P_{1}}_{a_{1}\text{-times}}, \underbrace{P_{2},\cdots,P_{2}}_{a_{2}\text{-times}},\cdots, \underbrace{P_{\ell},\cdots,P_{\ell})}_{a_{l}\text{-times}}$ is isomorphic over $\mathbb{P}^n$ to a permutation of $ (\underbrace{\psi^{\ast}P_{1},\cdots, \psi^{\ast}P_{1}}_{a_{1}\text{-times}}, \underbrace{\psi^{\ast}P_{2}, \cdots, \psi^{\ast}P_{2}}_{a_{2}\text{-times}}, \\\cdots, \underbrace{\psi^{\ast}P_{\ell},\cdots,\psi^{\ast}P_{\ell})}_{a_{l}\text{-times}}.$
Then, we have a short exact sequence of groups 
$$ 1 \,\longrightarrow\, \prod_{i}(\textnormal{Aut}_{\mathbb{P}^n} P_{i})^{a_{i}} \rtimes S_{a_{i}}) \,\longrightarrow\, \textnormal{Aut} \medspace P \,\longrightarrow\, H \,\longrightarrow\, 1.$$ Here the action of $S_{a_{i}}$ on $ (\textnormal{Aut}_{\mathbb{P}^n}P_{i})^{a_{i}}$ is by permutation of the factors.   
   
\end{manualtheorem}

\begin{manualtheorem}{D}\label{autsym}
Let $ n\,\geq\, 1$, and $ r\,\geq\, 2$ be integers, and let $P$ be a proper projective bundle over $ \mathbb{P}^n.$ Then there is an isomorphism $$ \textnormal{Aut}(S^{r}_{\mathbb{P}^n}P) \,\cong\, \textnormal{Aut}_{\mathbb{P}^n\rightarrow \mathbb{P}^n}(P).$$ 
\end{manualtheorem}

If condition $(1)$ or $(2)$ in Theorem \ref{autmult} holds, then $P$ is either a product of projective spaces or a product of projective spaces and $\mathbb{P}_{\mathbb{P}^n}(T_{\mathbb{P}^n})$. In this case, the automorphism group of $P$ is easy to describe (see Supplement \ref{supp}). 

Also note that one can further describe the group $ \textnormal{Aut}_{\mathbb{P}^n\rightarrow \mathbb{P}^n}(P)$ of Theorem \ref{autsym} by the following short exact sequence 
    $$ 1\,\longrightarrow\, \textnormal{Aut}_{\mathbb{P}^n}P \,\longrightarrow\, \textnormal{Aut}_{\mathbb{P}^n\rightarrow \mathbb{P}^n}(P) \,\longrightarrow\, H \,\longrightarrow\, 1,$$ 
    where 
    $$ H\,=\, \{ \psi \,\in\, \textnormal{Aut } \mathbb{P}^n \medspace | \medspace \psi^{\ast}P \,\cong\, P \hspace{2pt} \textnormal{over } \hspace{2pt} \mathbb{P}^n \}.$$
\section{Multiprojective bundles}
In this Section, we prove Theorems \ref{A} and \ref{autmult}, which we deduce from Theorem \ref{auto}. We begin by proving the following lemmas.
\begin{lemma} \label{contr}
Let $X$ be a variety, and let $\underline{E}\,=\,(E_{1},\cdots,E_{r })$ be an $r$-tuple of vector bundles over $X$. Suppose $Y$ is a variety and we have the following commutative diagram: \begin{center}
\begin{tikzcd}
\mathbb{P}_X(\underline{E}) \arrow[r, "f" ]  \arrow[d, "p"]
& Y \arrow[ld, "q"] \\
X
\end{tikzcd},
\end{center}
where $f$ and $q$ are proper maps, and $f_{\ast}\mathcal{O}_{{\mathbb{P}}_{X}(\underline{E})} \,=\, \mathcal{O}_{Y}$. Then, up to a permutation of $E_{1},\cdots,E_{r}$, we have  $Y \,\cong\, \mathbb{P}_{X}(E_{1},\cdots, E_{s})$ over $X$, for some $0\leq s\leq r$, $f,q$ being the natural projections.
\end{lemma}
\begin{proof}
Let $H$ be a $q$-ample line bundle on $Y$. Let $p_i\,:\,\mathbb{P}_{X}(\underline{E})\to \mathbb{P}_{X}(E_i)$ denote the projection maps. Then $f^{\ast}H \,\in\, $ Pic($\mathbb{P}_{X}(\underline{E})) \,=\,$ Pic$(X) \bigoplus (\oplus_{i\,=\,1}^{r}\mathbb{Z})$. Let $0\leq s\leq r$ be an integer. Then we can write:
\begin{equation*}
 f^\ast H \,=\, p^{\ast} L \otimes \left(\bigotimes_i p_i^*\mathcal{O}_{\mathbb{P}(E_{i})}(b_{i})\right),
\end{equation*} 
for a line bundle $L\in \textnormal{Pic}(X)$ and integers $b_{i}$. Since $f^{\ast}H$ is nef on each of the fiber of $p$, it follows that $b_{i}\,\geq\,0$ for all $i$. Up to a permutation of the $E_i$'s, we may assume $b_{1},\cdots,b_{s}\, >\, 0$, and $b_{s+1} \,=\, \cdots \,=\, b_{r} \,=\, 0.$ Hence, the morphism $f$ factors through $\mathbb{P}(E_{1},\cdots,E_{s}),$ as shown in the following diagram:  
\begin{center}
\begin{tikzcd}
 \mathbb{P}(\underline{E}) \arrow[rd,swap]{ur}[swap]{f} \arrow[d, ""] 
    \\ \mathbb{P}(E_{1},\cdots,E_{s}) \arrow[r,"f_{1}"] \arrow[d, "\pi" ]& Y  \arrow[ld, " "]  \\
    X
\end{tikzcd}
\end{center}
Let $\pi_i\,:\,\mathbb{P}_{X}(E_1, E_2,...,E_s)\to \mathbb{P}_{X}(E_i)$ be the projections for $1\leq i\leq s.$ Then
\begin{equation*}
f_{1}^{\ast}H \,=\, \pi^{\ast}L\otimes \left(\bigotimes_{i \,=\, 1}^{s}\pi_i^*\mathcal{O}_{\mathbb{P}(E_{i})}(b_{i})\right), 
\end{equation*}
which is $\pi$-ample. Therefore, $f_{1}$ is finite. Moreover, we are given that $f_{1}$$_{\ast}\mathcal{O}_{\mathbb{P}(E_{1},\cdots,E_{s})}=\mathcal{O}_Y.$ Hence, $f_{1}$ is an isomorphism.
\end{proof}
\begin{lemma}\label{setsum}
Let $A_1,A_2\subset\mathbb{C}$ be sets such that $|A_1|\,=\,|A_2|\,=\,2$, and let $l_i$ be the line in $\mathbb{C}$ that contains $A_i.$ Suppose that $A_1 +A_2$ lies on a circle $C$ in the complex plane. Let $\Bar{A_i}$ be the average of the elements of $A_i$. Then, the following holds:
\begin{enumerate}
\item[$(1)$] $l_1, l_2$ are not parallel.
\item[$(2)$] $\Bar{A_1}+ \Bar{A_2}$ is the centre of $C.$
\end{enumerate}
\end{lemma}
\begin{proof}
{(1):} If $l_1, l_2$ were parallel, then $A_1 +A_2$ would be contained in the line $l_1+l_2$ and would contain at least $3$ points. This contradicts the fact that $A_1 +A_2$ lies on the circle $C$.

\noindent{(2):} The points in $A_1 +A_2$ forms the $4$ vertices of a parallelogram $Q$. Since $A_1 +A_2\in C$, $Q$ must be a cyclic parallelogram, and thus a rectangle. Therefore, the centre of the circle $C$ is the centre of the parallelogram $Q$, which is $\Bar{A_1}+ \Bar{A_2}$.
\end{proof}
\begin{lemma}\label{alg}
Let $n\,\geq\, 2$ be an integer, and let $A$ \,=\,  $ \begin{pmatrix}
    a & b_{1} & b_{2} \\
    c_{1} & d_{11} & d_{12}\\
    c_{2} & d_{21} & d_{22}
\end{pmatrix}$ 
$\,\in\, \textnormal{GL}_{3}(\mathbb{C})$, $g_{1}\,\in\, \mathbb{C}[U_{1}], g_{2}\,\in\, \mathbb{C}[U_{2}]$ are monic quadratic polynomials, $G_{1} \,\in\, \mathbb{C}[U_{1}, T], G_{2}\,\in\, \mathbb{C}[U_{2}, T]$ be their homogenizations. Let $t, u_{1}, u_{2}$ be the images of $T, U_{1}, U_{2}$ in $ R\,=\,\frac{\mathbb{C}[T, U_{1}, U_{2}]}{(T^{n+1}, G_{1}(U_{1}, T), G_{2}(U_{2}, T))}.$ Define $x, y_1, y_2\,\in\, R$ by \begin{gather}
\begin{pmatrix}
    x \\
    y_{1}\\
    y_{2}
\end{pmatrix}
\,=\,
A
\begin{pmatrix}
    t \\
    u_{1}\\
    u_{2}
\end{pmatrix}.
\end{gather} Suppose there exist monic quadratic polynomials $f_{i}\,\in\, \mathbb{C}[Y_{i}]$, for $i \,=\, 1, 2$ such that if $F_{i}\,\in\, \mathbb{C}[Y_{i}, X]$ are their homogenizations then $x^{n+1} \,=\, 0$, $F_{1}(y_{1}, x) \,=\, 0,$ and $ F_{2}(y_{2}, x) \,=\, 0. $ Let $ \alpha_{1},\alpha_{2},\beta_{1},\beta_{2}$ be the averages of the roots of $g_1, g_2, f_1, f_2$ respectively. Assume that 
\begin{equation*}
a+b_1\alpha_1+b_2\alpha_2\,=\,0.
\end{equation*}
Then, either $b_1$ or $b_2$ is $0.$
\end{lemma}
\begin{proof}
Suppose $b_1b_2\neq0.$
By replacing $(T, U_{1},U_{2})$ with $ (T,U_{1}-\alpha_{1}T,U_{2}-\alpha_{2}T),$ and $ (X,Y_{1}, Y_{2})$ with $ (X,Y_{1}-\beta_{1}X ,Y_{2}-\beta_{2}X),$ we can assume that $a\,=\,0$, and
$$ g_{1}(U_{1})\,=\, U_{1}^{2}-a_{1}, \,\,\,\, g_{2}(U_{2})\,=\, U_{2}^{2}- a_{2},\,\,\,\, f_{1}(Y_{1})\,=\, Y_{1}^{2}- a_{3},\,\,\,\, f_{2}(Y _{2})\,=\, Y_{2}^2 - a_{4},$$
for some $ a_{1}, a_{2}, a_{3}, a_{4} \,\in\, \mathbb{C}$. This replacement does not change $b_{1},b_{2}.$ 

Since $F_{i}(y_{i},x)\,=\,0$ and $n\,\geq\, 2$, we see that  
\begin{equation*}
F_{i}(d_{i1}U_{1}+d_{i2}U_{2}+c_{i}T, b_{1}U_{1}+ b_{2}U_{2})
\end{equation*}
is a $\mathbb{C}$-linear combination of $G_{1}(U_{1},T)$ and $ G_{2}(U_{2}, T).$ Since the coefficients of $ U_{1}U_{2}, U_{1}T$ and $ U_{2}T$ are zero in $ G_{1}(U_{1},T)$ and $ G_{2}(U_{2},T),$ the coefficient of $ U_{1}U_{2}, U_{1}T$ and $U_{2}T$ are zero in 
 $$ F_{i}(d_{i1}U_{1}+d_{i2}U_{2}+c_{i}T, b_{1}U_{1}+ b_{2}U_{2}) \,=\, (d_{i1}U_{1}+ d_{i2}U_{2} + c_{i}T)^{2}- a_{2+i}(b_{1}U_{1}+b_{2}U_{2})^{2}.$$
This implies: 
\begin{equation*}
d_{i1} c_{i}\,=\, d_{i2}c_{i}\,=\,0, \quad
 d_{i1}d_{i2}\,=\,a_{2+i}b_{1}b_{2}.
\end{equation*}
If, for some $i$, we have $c_{i} \neq 0,$ then $d_{i1}\,=\,d_{i2}\,=\,0 $, and hence $a_{2+i}\,=\,0$ since $b_1b_2\neq 0.$ In that case, 
\begin{equation*}
F_{i}(c_{i}T, b_{1}U_{1}+b_{2}U_{2}) \,=\, c_{i}^{2}T^{2}.
\end{equation*}
It is a $\mathbb{C}$-linear combination of $G_{1}(U_{1},T)\,=\,U_{1}^{2}-a_{1}T^2$ and $G_{2}(U_{2},T)\,=\,U_{2}^{2}-a_{2}T^2.  $ But this is not possible unless $c_{i}\,=\,0,$ leading to a contradiction. Therefore, $c_{i}\,=\,0$ for $i \,=\, 1,2.$ This is also impossible as $A\,\in\, \textnormal{GL}_{3}(\mathbb{C})$ would then have one column zero. This contradiction proves the result.      
\end{proof}
The following lemma is crucial in the proof of Theorem \ref{auto}.
\begin{lemma}\label{factor}
Let $n,r,$ be positive integers, and let $\underline{E}\,=\,(E_{1},\cdots,E_{r })$, $\underline{F}\,=\,(F_{1},\cdots,F_{r })$ be $r$-tuples of vector bundles of rank $\,\geq\, 2$ over $\mathbb{P}^n$. Suppose $ \phi\,:\, \mathbb{P}(\underline{E})\to \mathbb{P}(\underline{F})$ is an isomorphism. Let $\mathbb{P}(\underline{E})\,\xrightarrow{\,\,\,f\,\,\,}\, \mathbb{P}^n$, $\mathbb{P}(\underline{F})\,\xrightarrow{\,\,\,h_i\,\,\,}\, \mathbb{P}(F_i)$ be the projections. Then there is an index $i$ such that the composition $f\circ \phi^{-1}$ factors through $h_i$.
\end{lemma}

\begin{proof}
Let $p_1,p_2,\cdots,p_r$ and $q_1, q_2,\cdots,q_r$ be positive integers such that rank $E_i\,=\,p_i+1$ and rank $F_i\,=\,q_i+1$ for all $i$. Set $ X\,=\, \mathbb{P}(\underline{F})$, and let $h\,:\,X\to \mathbb{P}^n, f_i\,:\,\mathbb{P}(\underline{E})\to \mathbb{P}(E_i)$ be the projections. Let $A^{\ast}(X)$ denote the Chow ring of $X$, and identify Pic$(X)=A^1(X)$. For a vector bundle $G$ over $X$, let $c(E)\in A^{\ast}(X)$ denote the total Chern class of $E$ (see \cite{fulton2013intersection} for definitions and results about the Chow ring and Chern classes). 
Let 
\begin{equation*}
t\,=\, h^*\mathcal{O}_{\mathbb{P}^n}(1),\quad x \,=\,  (f\circ \phi^{-1})^*\mathcal{O}_{\mathbb{P}^n}(1),\quad u_{i}\,=\,h_i^*\mathcal{O}_{\mathbb{P}(F_{i})}(1),\quad y_i\,=\,(f_i\circ \phi^{-1})^*\mathcal{O}_{\mathbb{P}(E_{i})}(1)
\end{equation*}
be elements of Pic$(X) \,=\, A^{1}(X).$ 
As in the proof of {\cite[Theorem 4.4]{BSV1}}, we have the following:
\\
\noindent$(1).$ There exist homogeneous polynomials $ g_{i}(U_i,T)\,\in\, \mathbb{Z}[U_{i},T]$ of degree $ q_{i}+1$, with the coefficients of $U^{q_{i}+1}$ equal to 1, such that $g_i(1,-t)\,=\,c(h_i^*F_i).$ The ring $A^{\ast}(X)$ is generated by $t,u_1,\cdots,u_r$ with only relations $t^{n+1}\,=\,g_i(u_i, t)\,=\,0$ for all $i.$\\\\
$(2).$ There exist homogeneous polynomials $ F_{i}(Y_i,X)\,\in\, \mathbb{Z}[Y_{i},X]$ of degree $ p_{i}+1$, with the coefficient of  $Y_i^{p_{i}+1}$ equal to 1, such that $f_i(1,-x)\,=\,c((f\circ\phi^{-1})^*E_i).$ The ring $A^{\ast}(X)$ is also generated by $x,y_1,\cdots,y_r$ with the only relations $x^{n+1}\,=\,F_i(y_i, x)\,=\,0$ for all $i.$

Twisting $F_i$'s by line bundles, we can also make the following assumption:

\paragraph{\textbf{Assumption ($\ast$)}} If for some $i$, we have $g_i(U_i, T)\,=\,(U_i-dT)^{q_i+1}$ for some integer $d$, then $d\,=\,0.$

There exist unique integers $a,b_{1},b_{2},\cdots,b_{r}$ such that $$ x\,=\, at+ \sum_{i\,=\,1}^{r} b_{i}u_{i}.$$ 

We assume without loss of generality that $b_{1},b_{2},\dots ,b_{l} \neq 0$,  and $b_{l+1}\,=\, \cdots \,=\,b_{r}\,=\,0,$ for some integer $0\,\leq\, l\,\leq\, r$. Thus, we have a commutative diagram 
\begin{center}
\begin{tikzcd}
 \mathbb{P}(\underline{E}) \arrow[r,swap]{ur}{\sim}[swap]{\phi} \arrow[dd, "f"] 
& \mathbb{P}(\underline{F}) \arrow[d, ""]    \\  & \mathbb{P}(F_{1},\cdots,F_{l}) \arrow[d,""] \arrow[ld, "\tau"]  \\
\mathbb{P}^n  & \mathbb{P}^{n} ,
\end{tikzcd}
\end{center}

Using Lemma \ref{contr} (with $Y\,=\,\mathbb{P}(F_{1},\cdots, F_{l})$), we may assume without loss of generality that we have a commutative diagram: 

\begin{center}
\begin{tikzcd}
 \mathbb{P}(E_{1},E_{2},\cdots, E_{l}) \arrow[r,swap]{ur}{\sim}[swap]{\phi} \arrow[d, ""] 
& \mathbb{P}(F_{1},\cdots,F_{l}) \arrow[d, ""]  \arrow[ld, "\tau"]  \\
\mathbb{P}^n  & \mathbb{P}^{n} ,
\end{tikzcd}
\end{center}
where $\tau$ does not factor through the projection 
\begin{equation*}
\mathbb{P}(F_{1},\cdots,F_{l})\to \mathbb{P}(F_{1},..,,F_{i-1},F_{i+1},\cdots,F_{l})
\end{equation*}
for any $1\,\leq\, i\, \leq \,l.$
We aim to show that $l\,\leq\, 1$. The above diagram shows that we may assume without loss of generality that $r\,=\,l$. So, assume that $b_{1},b_{2},\cdots ,b_{r}$ are all non zero. 
We have 
\begin{equation*}
A^{\ast}(\mathbb{P}_{\mathbb{P}^{n}}(\underline{E}))\,\cong\, A^{\ast}( \mathbb{P}^{n}\times\prod_{i\,=\,1}^{r}\mathbb{P}^{p_{i}}), \quad\textrm{and}\quad A^{\ast}(\mathbb{P}_{\mathbb{P}^{n}}(\underline{F}))\,\cong\, A^{\ast}(\mathbb{P}^{n}\times\prod_{i\,=\,1}^{s}\mathbb{P}^{q_{i}} ) 
 \end{equation*}
 as graded groups. By {\cite[Lemma 4.1]{BSV1}}, we see that the $p_i$'s are a permutation of $q_i$'s. Without loss of generality, assume $p_i\,=\,q_i$ for all $i$. 

Suppose $r\,\geq \,2$. We aim for a contradiction. Write 
$$ (aT+\sum_{i}b_{i}U_{i})^{n+1} \,=\, \sum_{i}Q(\underline{U},T)g_{i}(U_{i},T) + R(\underline{U},T),$$ where each
$Q_{i},g_{i},R$ is homogeneous of degree $n+1$, and $ \textnormal{deg}_{U_{i}} R\, \leq\, p_{i}$ for all $i.$
Since the set
\begin{equation*}
\{t^{\alpha }{u_{1}^{\alpha_{1}}\cdots}u_{r}^{\alpha_{r}}\,|\,  0\,\leq\, \alpha\,\leq\, n,\, \textnormal{and} \,\ \medspace 0\,\leq\,\alpha_{i}\,\leq\, q_{i}\} 
\end{equation*}
is linearly independent in $A^*(X)$, it follows that $ R(\underline{U},T)\,=\, \alpha T^{n+1}$ for some $ \gamma \,\in\, \mathbb{Z}$. For each $1\,\leq\, i\, \leq\, r,$ let
\begin{equation*}
Z_i\,=\,\{z\,\in\, \mathbb{C}\,|\, g_i(z,1)\,=\,0\},  \, \text{and}\, \, \Bar{Z_i}\,\,  \text{the average of the elements of } Z_i.
\end{equation*}
\underline{Case 1}: $ \gamma \,=\,0.$

We have 
$$\left(\sum_{i}b_{i}U_{i}+a\right)^{n+1}\,=\, \sum_{i}Q_{i}(\underline{U},1)g_{i}(U_{i},1) . $$

Let $z_i\,\in\, Z_i$ for $2\,\leq\, i\, \leq\, r.$ We then get 
\begin{equation*}
g_{1}(U_{1},1) \,\big{\vert}\, \left(b_{1}U_{1}+\sum_{i\,>\,1}b_{i}z_{i}+a\right)^{n+1},
\end{equation*}
so $ g_{1}(U_{1},1)\,=\, (U_{1}+ \lambda)^{p_{1}+1}$ for some $ \lambda \,\in\, \mathbb{Z}.$ Hence by assumption$(\ast)$,  we have 
\begin{equation*}
g_{1}(U_{1},1)\,=\,U_{1}^{p_{1}+1}.
\end{equation*}

Similarly, $ g_{i}(U_{i},1)\,=\, U_{i}^{p_{i}+1}$ for all $i$. Thus,  

$$ \left(\sum_{i}b_{i}U_{i}+a\right)^{n+1}\,=\, \sum_{i}Q_{i}(\underline{U},1) U_{i}^{p_{i}+1}.$$ Putting $ U_{i}\,=\,0$ for all $i$, we obtain $a\,=\,0$. So,  
$$ \left(\sum_{i}b_{i}U_{i}\right)^{n+1} \,\in\, (U_{1}^{p_{1}+1},\cdots, U_{r}^{p_{r}+1}).$$ As all $b_i$'s are nonzero, we conclude that $\sum_{i}p_{i}\, \leq\, n. $

Now write $ y_{i}\,=\, c_{i}t + \sum _{j} d_{i,j}u_{j} $, where $c_i, d_{i,j}\,\in\, \mathbb{Z}.$ Then, $ F_{i} (c_{i}T+ \sum_{j}d_{i,j}U_{j}, \sum_{i}b_{i}U_{i}) \,\in\, (T^{n+1} , U_{1}^{p_{1}+1}, \dots ,U_{r}^{p_{r}+1})$ for all $i$.
Since $ r\,\geq\, 2, $ we have $ p_{i}\,<\, \sum_{j}p_{j}\, \leq\, n.$ So, coefficient of $ T^{p_{i}+1} $ in $ F_{i}(c_{i}T+\sum_{j}d_{i,j}U_{j}, \sum_{i}b_{i}U_{i}) $  must be zero. Since $F_{i}(Y_{i},X)$ is homogeneous of degree $ p_{i}+1$ and coefficient of $ Y_{i}^{p_{i}+1} $ is 1, we get $c_{i}^{p_{1}+1}\,=\,0, $ so $ c_{i}\,=\,0.$  If 

 \[ A\,=\,\begin{bmatrix}
a &b_{1}&\cdots &b_{r} \\
c_{1}& & &  \\
\vdots &   & d_{i,j} &    \\
c_{r}&  &  & 
\end{bmatrix}\] then 
\begin{equation}
 \begin{bmatrix}
x\\
y_1\\
y_2\\
\vdots\\
y_r
\end{bmatrix}
\,=\, A  \begin{bmatrix}
t\\
u_1\\
u_2\\
\vdots\\
u_r
\end{bmatrix}.
\end{equation}
So, $ A \,\in\, \textrm{GL}_{r+1}(\mathbb{Z}).$ But $A$ has one column completely zero, a contradiction.  

\underline{Case 2 :} $ \gamma \neq 0$

We have 
\begin{equation}\label{eq}
    (\sum_{i}b_{i}U_{i}+a)^{n+1} \,=\, \sum_{i} Q_{i}(\underline{U},1) g_{i}(U_{i},1) +\alpha.
\end{equation}

Let $z_i\,\in\, Z_i$ for $1\,\leq\, i\, \leq\, r,$ then putting $ U_{i}\,=\,z_{i}$ for $i\, \geq \,2  $ we get $ g_{1}(U_{1},1) \big{\vert} (b_{1}U_{1}+ \sum_{i\,>\,1} b_{i})^{n+1}-\alpha .$ As $ \alpha \neq 0$, $ (b_{1}U_{1}+ \sum_{i\,>\,1}b_{i}z_{i}+a)^{n+1} $ has distinct  roots. So, $ g_{1}(U_1, 1)$ has distinct roots. Similarly, each $g_{i}(U_i,1)$ has distinct roots. 

Putting $U_i\,=\,z_i$ for all $i$ in \eqref{eq}, we get $(\sum_i b_{i}z_{i}+a)^{n+1}\,=\,\alpha . $

Let $ C$ be the circle $\{ \omega\,\in\, \mathbb{C} \medspace|\medspace  \omega^{n+1} \,=\, \alpha \}.$ So we have 
\begin{equation}\label{sumbizi}
    \sum_i b_{i}Z_{i}+a\,\in\,  C. 
\end{equation}

\underline{Claim 1}: $p_{i}\,=\,1 $ for all $i$

    \begin{proof}
        
     Suppose not, say $ p_{1}\, \geq\, 2 .$ Fix $ z_i\,\in\, Z_i$ for $3\,\leq\, i\, \leq\, r,$ fix $ z'_{2}\neq z_{2}$ in  $Z_2$ (possible as we have  deg $g_{2}(U_2,1)\,=\, p_{2}+1 \,\geq\, 2$ and $ g_{2}(U_2,1) $  has distinct roots). Let $ A_{1}\,=\, b_{1}Z_{1}+ b_{2}z_{2}+ \dots + b_{r}z_{r}, A_{2} \,=\, b_{1}Z_{1}+ b_{2}z'_{2}+ \dots + b_{r}z_{r}.$ So we have 
    $ A_{2} \,=\, A_{1} + b_{2}(z'_{2}- z_{2}).$
   So we have $ A_{1},A_{2}\subseteq C$, so $ A_{1} \subset C \cap [ C- b_{2}(z'_{2}-z_{2})].$ Since $ b_{2}(z_{2}^{'}-z_{2}) \neq 0$, $C$ and $C-b_{2} (z'_{2}-z_{2}) $ are distinct circles. Two distinct circles can intersect in at most two points, so $|A_{1}|\,\leq\, 2.$ But this contradicts $$ |A_{1}| \,=\,  \textnormal{number of roots of} \medspace g_{1}(U_1,1)  \,=\, p_{1}+1\, \geq\, 3. $$  This proves the claim. 
\end{proof}
So, $|Z_i|\,=\,2$ for all $i.$
For $1\,\leq\, i\, \leq\, r,$ let $l_i$ be the line in $\mathbb{C}$ joining the elements of $Z_i$.

\underline{Claim 2:} $r\,=\,2$, one of $g_1(U_1,1), g_2(U_2,1)$ has real roots, other has non-real roots.

\begin{proof}
    Since $g_i(U_i,1)$ has real coefficients, $l_i$ is either the real line or a line perpendicular to the real line. It suffices to show that for $i\neq j$, $l_i$ and $l_j$ are not parallel.
 Fix $z_i\,\in\, Z_i$ for $3\,\leq\, i\, \leq\, r.$ We have $b_1 Z_1+b_2Z_2\,\in\, C-\sum_{3\,\leq\, i\,\leq \,r}b_{i}z_{i} -a .$ By Lemma \ref{setsum} $(1)$, We get $l_i,l_j$ are not parallel.
 \end{proof}
Say $ g_{1}(U_1,1)$ has non-real roots. If $ n\,=\,1,$ $F_{1}$ is split, so $g_{1}(U_1,1)$ has both roots real $ ( F_{1} \,=\, \mathcal{O}(a) \oplus \mathcal{O}(b) \Rightarrow g_{1}(U_{1},1) \,=\, (U_{1}-a) (U_{2}-b))$. This is a contradiction. So we have $ n\, \geq\, 2.$     

Let $ A \,\in\, \textrm{GL}_{3}(\mathbb{Z})$ be such that 
\begin{equation*}
 \begin{bmatrix}
x\\
y_1\\
y_2\\
\end{bmatrix}
\,=\, A  \begin{bmatrix}
t\\
u_1\\
u_2\\
\end{bmatrix}.
\end{equation*}
\eqref{sumbizi} and Lemma \ref{setsum} applied to $b_1Z_1, b_2Z_2$ shows $a+ b_1\Bar{Z_1}+ b_2\Bar{Z_2}\,=\,0.$
Now Lemma \ref{alg}, gives a contradiction. 
\end{proof}
\begin{theorem} \label{auto}
Let $n,r$ be positive integers, and let $\underline{E}\,=\,(E_{1},\cdots,E_{r })$, $\underline{F}\,=\,(F_{1},\cdots,F_{r })$  be $r$-tuples of vector bundles of rank $\,\geq\, 2$ over $\mathbb{P}^n$. Suppose $ \phi\,:\, \mathbb{P}(\underline{E})\to \mathbb{P}(\underline{F})$ is an isomorphism.
Then $ \phi$ is over an automorphism of $ \mathbb{P}^n$, unless one of the following holds:
 \begin{enumerate}
\item[$(1)$] All $E_{i}, F_{j}$'s are trivial up to line bundle twists.
\item[$(2)$] Up to a permutation of $ E_{i}$'s and a permutation of $F_{j}$'s, and line bundle twists, we have $ E_{1} \,\cong\, F_{1} \,\cong\, T_{\mathbb{P}^n},$ and all $ E_{i},F_j$'s are trivial for  i,\, j \,>\,1.
\end{enumerate}
\end{theorem}
\begin{proof}
Let $p_1,p_2,\cdots,p_r$ and $q_1, q_2,\cdots,q_r$ be positive integers such that rank $E_i\,=\,p_i+1$, and rank $F_i\,=\,q_i+1$ for all $i$.
Suppose $ \phi$ is not an automorphism over $ \mathbb{P}^n.$ We show that either all $F_i$'s are trivial up to line bundle twists, or up to a permutation of $F_i$'s, we have $F_{1} \,\cong\, T_{\mathbb{P}^n},$ and all $F_i$'s are trivial up to line bundle twists for $ i\,>\,1.$ The analogous statement for $E_i$'s would be followed by symmetry or by a similar argument.
By Lemma \ref{factor}, we can assume without loss of generality that we have a commutative diagram 
\begin{center}
\begin{tikzcd}
 \mathbb{P}(\underline{E}) \arrow[r,swap]{ur}{\sim}[swap]{\phi} \arrow[dd, "f"] 
& \mathbb{P}(\underline{F}) \arrow[d, "h_{1}"]    \\  & \mathbb{P}(F_{1}) \arrow[d,"\pi_1"] \arrow[ld, "\tau"]  \\
\mathbb{P}^n  & \mathbb{P}^{n}
\end{tikzcd}.
\end{center}
By Lemma \ref{contr} applied to $Y\,=\,\mathbb{P}(F_1)$,  $\tau$ is a projective bundle. The map $\tau$ does not factor through $\pi_1$, as $ \phi$ is not over an automorphism of $ \mathbb{P}^n.$ So $ \mathbb{P}(F_{1})$ has two different projective bundle structure over $ \mathbb{P}^n.$ Then, by \cite{Sa}, two cases can happen:

\underline{Case 1:} $F_{1}\,=\,\mathcal{O}_{\mathbb{P}^n}^{n+1}$, up to a line bundle twist. 

So, $\tau\,:\,(\mathbb{P}^n)^2\to \mathbb{P}^n$ is projection onto one of the factors. For $z\,\in\, \mathbb{P}^n$, we have $f^{-1}(z) \,\cong\, h_{1}^{-1} \tau^{-1}(z), $ so 

  $$ \prod_{i} \mathbb{P}^{p_{i}} \,\cong\, \mathbb{P}_{\mathbb{P}^n}(F_{2},\cdots , F_{r}).$$ So, each $ F_{i}$ for $i\, \geq\, 2$ is trivial up to line bundle twists by {\cite[Theorem 3.17]{BSV1}}.   

  \underline{Case 2} : $F_{1} \,=\, T_{\mathbb{P}^n}$ up to line bundle twist. 

So, if we can identify $\mathbb{P}_{\mathbb{P}^n}(F_1)$ as a divisor in $\mathbb{P}^n\times \mathbb{P}^n$ of degree $(1,1)$ such that $\tau$ and $\pi_1$ are the projections onto the two factors. Hence, 
$$  \{ \pi_1( \tau^{-1}(z)) \medspace | \medspace z \,\in\, \mathbb{P}^n \}$$ is the set of all hyperplanes in $ \mathbb{P}^n  $. For $z\,\in\, \mathbb{P}^n,$ we have $ f^{-1}(z) \,\cong\, h_{1}^{-1} \tau^{-1}(z),$ so 

$$\prod_{i} \mathbb{P}^{p_i} \,\cong\, \mathbb{P}_{\pi_1( \tau^{-1}(z))}(F_{2},\cdots , F_{r}). $$ 
By {\cite[Theorem 3.17]{BSV1}} $,  F_i \mid_{\pi_1( \tau^{-1}(z))}$ is trivial up to line bundle twists, for all $ z \,\in\, \mathbb{P}^n$ and $i\,\geq \,2$. Hence, for each $i\,\geq\, 2$, $F_i$ restricted to any hyperplane in $\mathbb{P}^n$ is trivial up to line bundle twist. Looking at the first Chern class of $F_i$, we see that there is a line bundle $L_i$ on $\mathbb{P}^n$ such that the restriction of $F(L_i)$ to every hyperplane is trivial. By {\cite[Theorem 3.2.1]{Ok}}, $F_i(L_i)$ is trivial. So, $F_{2}, \dots F_{r}$ are trivial up to line bundle twists.
\end{proof}
Now we are ready to prove Theorems \ref{A} and \ref{autmult}, which will easily follow from Theorem \ref{auto} and the following Lemma.
 \begin{lemma} \label{iso}
      Let $X$ be a variety, $n,r \,\geq\, 1$, and  let $\underline{E}=(E_1, E_2,\cdots, E_r),$ $ \underline{F}=(F_1, F_2,\cdots, F_r)$ be $r$-tuples of vector bundles over $X$ of rank $\,\geq\, 2$. Let $\phi\,:\, \mathbb{P}_{X}(\underline{E}) \to \mathbb{P}_{X}(\underline{F})$ be an isomorphism over $X$. Then there is unique $\sigma\,\in\, S_r,$ and unique isomorphisms $\phi_i\,:\, \mathbb{P}_{X}(E_i) \to \mathbb{P}_{X}(F_{\sigma(i)})$ over $X$, such that $\phi$ is the composition $\mathbb{P}_{X}(\underline{E})\,\xrightarrow{\,\,\,\prod_i \phi_i\,\,\,}\, \mathbb{P}_{X}(\sigma(\underline{F}))\,\xrightarrow{\,\,\,T_{\sigma}^{-1}\,\,\,}\, \mathbb{P}_{X}(\underline{F}).$ Here $\sigma(\underline{F})\,=\,(F_{\sigma(1)}, F_{\sigma(2)},\cdots, F_{\sigma(r)})$, and $\mathbb{P}_{X}(\underline{F})\,\xrightarrow{\,\,\,T_{\sigma}\,\,\,}\, \mathbb{P}_{X}(\sigma(\underline{F}))$ is the permutation of the factors.
  \end{lemma}
\begin{proof}
We only need to show the existence, and the uniqueness is clear.    Let 
\begin{equation*}
\mathbb{P}_{X}(\underline{E})\,\xrightarrow{\,\,\,p_{i}}\mathbb{P}_{X}(E_{i}) \qquad \textrm{and} \qquad  \mathbb{P}_{X}(\underline{F})\,\xrightarrow{\,\,\,q_{i}\,\,\,}\,\mathbb{P}_{X}(F_{i}) 
\end{equation*}
be the projections. By Lemma \ref{contr}, $p_{i}$'s are exactly the contractions $\mathbb{P}_{X}(\underline{E})\to Y$ over $X$ such that the Picard rank of $Y$ is $ 2.$ An analogous statement holds for $q_{i}$'s. Since $\phi$ is an isomorphism over $X$, there is $\sigma\,\in\, \textrm{S}_{r}$ such that for all $i$, we have a commutative diagram with $\phi_{i}$ isomorphism:
\begin{center}
\begin{tikzcd}
\mathbb{P}_{X}(\underline{E}) \arrow[r, "\phi"] \arrow[d, "p_{i}"'] & \mathbb{P}_{X}(\underline{F}) \arrow[d, "q_{\sigma(i)}"] \\
\mathbb{P}_{X}(E_{i})\arrow[r, "\phi_{i}"'] \arrow[dr, "p"'] &\mathbb{P}_{X}(F_{\sigma(i)})  \arrow[d, "q"] \\
& X
\end{tikzcd}
\end{center}
The composition $\psi\,:\,\mathbb{P}_{X}(\underline{F})\,\xrightarrow{\,\,\,\phi^{-1}\,\,\,}\, \mathbb{P}_{X}(\underline{E})\,\xrightarrow{\,\,\,\prod_i \phi_i\,\,\,}\, \mathbb{P}_{X}(\sigma(\underline{F}))\,\xrightarrow{\,\,\,T_{\sigma}^{-1}\,\,\,}\, \mathbb{P}_{X}(\underline{F})$ is an automorphism of $\mathbb{P}_{X}(\underline{F})$ commuting with every projection $q_i$. So, for every $x\,\in\, X$, if $\psi_x$ denotes the restriction of $\psi$ to the fibre over $x$, then $\psi_x$ is an automorphism of the product of projective spaces $\prod_i \mathbb{P}(F_{i,x})$ commuting with projection to every factor. Thus, $\psi_x$ is identity for every $x$, hence $\psi$ is the identity map.
\end{proof}
\begin{proof}[Proof of Theorem \ref{A}]
Let $\mathbb{P}_{\mathbb{P}^{n}}(\underline{E})\,\xlongrightarrow{\phi}\mathbb{P}_{\mathbb{P}^{n}}(\underline{F})$ be an isomorphism. If $(1)$ or $(2)$ of Theorem \ref{auto} occurs, then clearly the conclusion of Theorem \ref{A} holds. Now, suppose neither $(1)$ nor $(2)$ of Theorem \ref{auto} holds. By Theorem \ref{auto}, $\phi$  is over an automorphism $\psi$ of $\mathbb{P}^{n}$. Therefore, we have the isomorphism $\tilde{\phi}\,:\, \mathbb{P}_{\mathbb{P}^{n}}(\underline{E})\to \mathbb{P}_{\mathbb{P}^{n}}(\psi^*\underline{F})$ over $\mathbb{P}^{n}$. Now, by Lemma \ref{iso}, there is $\sigma\,\in\, S_r$ such that $\mathbb{P}_{\mathbb{P}^{n}}(E)\,\cong\, \mathbb{P}_{\mathbb{P}^{n}}(F_{\sigma(i)})$ over $\mathbb{P}^n$ for all $i$. Hence, the conclusion of Theorem \ref{A} holds.
\end{proof}
        
\begin{proof}[Proof of Theorem \ref{autmult}]
Let $ \phi \,\in\, \textnormal{Aut}\medspace P$. By Theorem \ref{auto}, $ \phi$ is over a (necessarily unique) automorphism $ \psi$ of $ \mathbb{P}^n$. So, we have an isomorphism 
    $$ \widetilde{\phi}\,:\, P \,\longrightarrow\, \psi^{\ast} P$$
    over $ \mathbb{P}^n$. By Lemma \ref{iso}, $ \psi \,\in\, H.$ This gives a group homomorphism $ \textnormal{Aut} \medspace P \,\xrightarrow {\,\,\,\eta\,\,\,}\, H.$ For $ \psi \,\in\,H,$ we have $ P \,\cong\, \psi^{\ast}P$ over $ \mathbb{P}^n$, hence there is $ \phi \,\in\, \textnormal{Aut}\,{P}$ with $ \eta (\phi)\,=\, \psi.$ Therefore, $ \eta$ is surjective. We have $ \textnormal{Ker} \medspace \eta \,=\, \textnormal{Aut}_{\mathbb{P}^n}(P)$. We have inclusions 
    $\prod_{i}S_{a_{i}} \xhookrightarrow{j_{1}} \textnormal{Aut}_{\mathbb{P}^n}P$ given by permutations of the factors of $P$, and
    $ \prod_{i}(\textnormal{Aut}_{\mathbb{P}^n} P_{i})^{a_{i}}\xhookrightarrow{j_{2}} \textnormal{Aut}_{\mathbb{P}^n}(P) $
    given by products of automorphism of each factor of $ P.$ Let im($j_1$), and im($j_2$) be the images of $j_1$ and $j_2$, respectively. Our goal is to show that $\textnormal{Aut}_{\mathbb{P}^n}(P)=\textnormal{im}(j_2)\rtimes \textnormal{im}(j_1)$.
    
Since $P_i$'s are pairwise non-isomorphic over $ \mathbb{P}^n$, by Lemma \ref{iso}, any $ \phi \,\in\, \textnormal{Aut}_{\mathbb{P}^n}(P)$ can be written in a unique way as $ \phi_{1} \cdot \phi_{2},$ where $ \phi_{1} \,\in\, \textnormal{im}(j_{1})$, $ \phi_{2} \,\in\,  \textnormal{im}(j_{2}).$ Thus, $ \textnormal{im} (j_{1}) \cap \textnormal{im}(j_{2})\,=\,1,$ $ \textnormal{im}(j_{1}) \cdot \textnormal{im}(j_{2})\,=\, \textnormal{Aut}_{\mathbb{P}^n}(P).$ Also, any element of $ \textnormal{im}(j_{1})$ normalizes $ \textnormal{im}(j_{2})$. Hence  $ \textnormal{im}(j_{2})$ is a normal subgroup of $ \textnormal{Aut}_{\mathbb{P}^n}(P)$. This implies $\textnormal{Aut}_{\mathbb{P}^n}(P)=\textnormal{im}(j_2)\rtimes \textnormal{im}(j_1)$.
    
Hence, $$\textnormal{Aut}_{\mathbb{P}^n}(P)\,\cong\, \prod_{i}(\textnormal{Aut}_{\mathbb{P}^n}P_{i})^{a_{i}} \rtimes \prod_{i}S_{a_{i}  } \,\cong\, \prod_{i}((\textnormal{Aut}_{\mathbb{P}^n}P_{i})^{a_{i}} \rtimes S_{a_{i}}).$$ 
\end{proof}
Finally, we describe the Aut$(P)$ if $(1)$ or $(2)$ in Theorem \ref{autmult} holds.
\begin{supplement}
\label{supp}
Let $ r \geq 0$ be an integer,  $ m_{1}, m_{2}, \cdots m_{r}$ distinct positive integers and $ n\geq 2$, $ a_{1},\cdots a_{r}$ positive integers. Then: 
\begin{enumerate}
\item  \textnormal{Aut} $\prod_{i=1}^r(\mathbb{P}^{m_{i}})^{a_{i}} \cong \prod_{i=1}^{r} 
[(\textnormal{Aut}\,\mathbb{P}^{m_{i}})^{a_{i}} \rtimes S_{a_{i}}].$ 
\vspace{2pt}
\item $\textnormal{Aut}\big(\mathbb{P}_{\mathbb{P}^n}(T_{\mathbb{P}^n}) \times \prod_{i=1}^{r}(\mathbb{P}^{m_{i}})^{a_{i}}\big) \cong \big(\textnormal{Aut} \,\mathbb{P}^n \rtimes \mathbb{Z}/2 \mathbb{Z} \big)\times \prod_{i=1}^{r} [(\textnormal{Aut}\,\mathbb{P}^{m_{i}})^{a_{i}} \rtimes S_{a_{i}}].$
\end{enumerate}
\end{supplement}
\begin{proof}

The same proof as {\cite[Theorem 2.7, Corollary 2.8]{Ha}} implies the following (see also {\cite[Theorem 4.1]{Og}}): If
$Y = \prod _{i} Y_{i}^{a_{i}}$ where $Y_i$'s are smooth projective varieties with irregularity $0$ and none of them is a nontrivial product, then, $\textnormal{Aut}(Y) \cong \prod_{i}(\textnormal{Aut}(Y_{i})^{a_{i}}\rtimes S_{a_{i}})$.

So, we only need to show  $$ \textnormal{Aut} (\mathbb{P}_{\mathbb{P}^n}(T_{\mathbb{P}^n}))\cong \textnormal{Aut}({\mathbb{P}^n}) \rtimes \mathbb{Z}/2 \mathbb{Z}.$$ Let $G= \textnormal{Aut} (\mathbb{P}(T_{\mathbb{P}^n}))$ any automorphism of $ \mathbb{P}^n$  induces an automorphism of $ \mathbb{P}_{\mathbb{P}^n}(T_{\mathbb{P}^n}),$ so we have inclusion 
    $ i: \textnormal{Aut}(\mathbb{P}^n) \hookrightarrow G .$ We can identify $ \mathbb{P}(T_{\mathbb{P}^n})$ as
      $$ D= \{ ([\underline{x}],[\underline{y}]) \in \mathbb{P}^n \times \mathbb{P}^n \medspace | \medspace \underline{x}^t \underline{y}=0 \}.$$ Swap the two factors of $ \mathbb{P}^n \times \mathbb{P}^n$ gives an order two automorphism $\eta$ of $ \mathbb{P}(T_{\mathbb{P}^n}),$ hence we have a subgroup $ H \leq G$ isomorphic to $ \mathbb{Z}/2\mathbb{Z}.$ Clearly, $H$ normalizes $ \textnormal{im}(i)$. It suffices to show $ G= H \cdot \textnormal{im}(i).$ 
      
      Let $ \phi \in G.$ We want to show $ \phi \in H \cdot \textnormal{im}(i).$ Composition by $ \phi$ would permute the elementary contractions of $ \mathbb{P}_{\mathbb{P}^n}(T_{\mathbb{P}^n}),$ so, composing $\phi$ by $ \eta $ if needed, we can assume that there are $ \psi_{1}, \psi_{2}\in \textnormal{Aut} (\mathbb{P}^n)$  such that following diagram
       \begin{center}
\begin{tikzcd}
D \arrow[r, "\phi"] \arrow[d,"\pi_{i}"]
& D \arrow[d, "\pi_{i}" ] \\
\mathbb{P}^n  \arrow[r,"\psi_{i}" ]
& |[, rotate=0]| \mathbb{P}^n
\end{tikzcd}
\end{center}
commutes, $\pi_{1}, \pi_{2}$, are projection onto two factors. Composing $ \phi$ by an element of $ \textnormal{im}(i)$, we may assume $ \psi_{1}=id.$ There is $ (n+1) \times (n+1)$ invertible matrix $A_{2}$ over $ \mathbb{C}$ such that $ \psi_{2}([\underline{y}])= [A_{2}\underline{y}]$ for all $ \underline{y} \in \mathbb{P}^n$. So, 
$$ \phi([\underline{x}],[ \underline{y}])= ( [\underline{x}], [A_{2}\underline{y}]) \quad \text{ for all } \medspace [\underline{x}], [\underline{y}] \in \mathbb{P}^n.$$
Hence, if $ \underline{x}, \underline{y}\neq 0$  with $ \underline{x}^{t}\underline{y}=0$, we have $ \underline{x}^{t} A \underline{y}=0.$ This forces $A$ to be a scalar matrix, so $ \psi_{2}= id.$ So, $ \phi= id$, and we are done.
\end{proof} 
\section{Symmetric powers of projective bundles}
In this Section, we prove Theorems \ref{sym} and \ref{autsym}.
First, we need the following Lemma. For a projective variety $Y$, we denote the Picard rank of $Y$ by $\rho(Y).$ For a normal variety $X$, let Cl$(X)$ denote the divisor class group of $X.$
 \begin{lemma}\label{picard}
Let $ p,r,n$ be positive integers, $r\geq 2$, $E$ a vector bundle of $ p+1$ over $\mathbb{P}^n.$ Then $ \rho(S_{\mathbb{P}^n}^{r}\mathbb{P}(E))\,=\,2.$ 
\end{lemma}
      
\begin{proof}
Let $H \hookrightarrow \mathbb{P}^n$ be a hyperplane, $U\,=\, \mathbb{P}^n \backslash  H \,\cong\, \mathbb{A}^n.$ Let $ X\,=\, S_{\mathbb{P}^n}^{r} \mathbb{P}(E) \,\xrightarrow{\,\,\,\pi\,\,\,}\, \mathbb{P}^n$ be the projection. Any vector bundle on $\mathbb{A}^n$ is trivial, so $ \pi^{-1} U \,\cong\, S^{r} \mathbb{P}^p \times \mathbb{A}^n$. By {\cite[Proposition 6.6, Chapter 2]{Hart}} and {\cite[Theorem 1(ii)]{BansalSarkarVats2025_SymmetricPower}}, we have $$ \textnormal{Cl}(\pi^{-1}{U}) \,\cong\, \textnormal{Cl} (S^{r} \mathbb{P}^{p}) \,\cong\, \mathbb{Z} \oplus \mathbb{Z}/2 \mathbb{Z}.$$ Since $ \pi^{-1} H\,\cong\, S_{H}^{r} \mathbb{P}(E\mid_{H})$ is irreducible, by {\cite[Proposition 6.5, Chapter 2]{Hart}} we have an exact sequence 
$$ \mathbb{Z} \,\longrightarrow\, \textnormal{Cl}( X) \,\longrightarrow\, \textnormal{Cl}(\pi^{-1}(U)) \,\longrightarrow\, 0, $$ 
which is also left exact as $X$ is a complete variety (hence no nonzero multiple of $\pi^{-1}H$ can be linearly equivalent to 0).  
 So, rank $ \textnormal{Cl}(X) \,=\,2.$ By  {\cite[Lemma 5.16]{Wil}}, $ X$ is $ \mathbb{Q}$-factorial. Hence, $ \rho(X)\,=\,2.$
\end{proof}
Now we can prove Theorems \ref{sym} and \ref{autsym}.
\begin{proof}[Proof of Theorem \ref{sym}]
Let $ Z\,=\, S_{\mathbb{P}^n}^{r} \mathbb{P}(E)$,  $ W \,=\, S^{s}_{\mathbb{P}^m} \mathbb{P}(F)$, $ Z \,\xrightarrow{\,\,\,\pi_{1}} \mathbb{P}^n\,\,\,$\,, $ W \,\xrightarrow{\,\,\,\pi_{2}} \mathbb{P}^m\,\,\,$\, be the projections. So, $ \pi_{1}^{-1}(x) \,\cong\, S^{r}\mathbb{P}^{p}$, $ \pi_{2}^{-1}(y)\,\cong\, S^{s}\mathbb{P}^q $ for all $x \,\in\, \mathbb{P}^n, y \,\in\, \mathbb{P}^m.$ Let $X,Y$ be the images of the diagonal embeddings of $ \mathbb{P}(E)$, $ \mathbb{P}(F)$ in $ Z, W$ respectively. 

For a closed subset $N$ of a variety $M$, denote the codimension of $N$ in $M$ by codim$_M N$. For a reduced scheme $M$ of finite type over $\mathbb{C}$, let Sing $M$ be the singular locus of $M$ with reduced subscheme structure. For nonnegative integers $t$, define Sing$^t(M)$ inductively by Sing$^0(M)\,=\,M$, and Sing$^t(M) = $ Sing(Sing$^{t-1} M)$ for $t\,\geq\, 1$. By {\cite[Theorem 2]{BansalSarkarVats2025_SymmetricPower}}, for each $ x \,\in\, \mathbb{P}^n$ we have $$\textnormal{codim}_{\pi_1^{-1}(x)} \textnormal{Sing}(\pi_{1}^{-1}(x))\,=\,p, \quad\textnormal{Sing}^{r-1} \pi_{1}^{-1}(x)\,=\, X\cap \pi_{1}^{-1}(x), \quad \textnormal{Sing}^{r} \pi_{1}^{-1}(x)\,=\, \varnothing .$$ Since $ \pi_{1}$ is locally a product, it follows that: 
  $$ \textnormal{ codim}_{Z} \textnormal{Sing}Z\,=\, p, \quad \textnormal{Sing}^{r-1} Z\,=\, X, \quad \textnormal{Sing}^r Z\,=\, \varnothing. $$
  Similarly, $ \textnormal{codim}_{W} \textnormal{Sing}\,W\,=\, q$, $ \textnormal{Sing}^{s-1}W\,=\,Y$, $ \textnormal{Sing}^{s}W \,=\, \varnothing$.  So, 
 \begin{equation}\label{Bequ1} 
   p\,=\,q, r\,=\,s,
  \end{equation}
  and 
  \begin{equation}\label{Bequ2}
 \phi(X)\,=\,Y.
\end{equation}
Also, $ n+rp\,=\, \textnormal{dim}\, Z\,=\, \textnormal{dim}\, W\,=\, m+sq,$ So, \eqref{Bequ1} gives 
\begin{equation}\label{Bequ3}
    m\,=\,n.
\end{equation}

By Lemma \ref{picard}, the Picard rank of $Z$ is $2.$ It suffices to show that $ \pi_{1}\, :\, Z \,\longrightarrow\, \mathbb{P}^n$, $ \pi_{2} \circ \phi\,:\, Z \,\longrightarrow\, \mathbb{P}^n$ are the same contraction, as then $\phi $ would be over an automorphism $ \psi$ of $ \mathbb{P}^n$, and \eqref{Bequ2} will give  $ \mathbb{P}(E) \,\cong\, \mathbb{P}(\psi^{\ast}F)$ over $ \mathbb{P}^n,$ hence $ E \,\cong\, \psi^{\ast}F \otimes L$ for a line bundle $L$ on $ \mathbb{P}^n.$   
 
 Suppose $p_1$ and $\pi_2\circ \phi$ are different contractions of $Z$. By Lemma \ref{picard}, $ \rho(Z)\,=\,2.$ So, $p_1$ and $\pi_2\circ \phi$ are the contractions of the two rays of the Mori cone $\overline{\textnormal{NE}}(Z)$. Hence, no curve in $Z$ can be contracted by both $p_1$ and $\pi_2\circ \phi$. This means the map $$ (\pi_1, \pi_2\circ \phi): Z \,\longrightarrow\, \mathbb{P}^n \times \mathbb{P}^n$$
is finite. Hence $ n+rp\,=\, \textnormal{dim}\, Z\, \leq\, \textnormal{dim}\,(\mathbb{P}^n \times \mathbb{P}^n)\,=\, 2n.$ So, 
  \begin{equation}\label{Bequ4}
   rp\, \leq\, n. 
  \end{equation}
  Also, $ \pi_{1}{\mid_{X}}$, $ \pi_{2} \circ \phi{\mid_{X}}$ are two different contractions of $ X\,=\, \mathbb{P}(E).$ So, $X$ has two different $\mathbb{P}^p$-bundles structures over $\mathbb{P}^n$. By \cite{Sa}, we get 
  \begin{equation}\label{Bequ5}
  p\,\geq\, n-1.
  \end{equation}
  So, \eqref{Bequ4} and \eqref{Bequ5}
 give $  n\, \geq\, rp\, \geq\, r(n-1) \,\geq\, 2(n-1)$, hence $ n\,\leq\, 2$. As $ r\, \geq\, 2$, \eqref{Bequ4} gives $ p\,\leq\, 1,$ contradiction.
\end{proof}
\begin{proof}[Proof of Theorem \ref{autsym}]        
    We have an injective group homomorphism 
    $$ \textnormal{Aut}_{\mathbb{P}^n \rightarrow \mathbb{P}^n}(P) \,\xrightarrow{\,\,\,j\,\,\,}\, \textnormal{Aut}(S^{r}_{\mathbb{P}^n}P) $$ given by $$ \phi \longmapsto S_{\mathbb{P}^n}^{r}(\phi).$$ 
We want to show that $j$ is surjective. Let $ \phi \,\in\, \textnormal{Aut}(S_{\mathbb{P}^n}^{r}P).$ By Theorem \ref{sym}, $ \phi$ is over an automorphism of $ \mathbb{P}^n.$ Let $ X \subset S^{r}_{\mathbb{P}^n}P$ be the image of the diagonal embedding of $P$. As we saw in the proof of the Theorem \ref{sym}, $ X\,=\, \textnormal{Sing}^r ( S_{\mathbb{P}^n}^rP).$ So, $ \phi(X)\,=\,X$, hence $ \phi{\mid_{X}}\,\in\, \textnormal{Aut}_{\mathbb{P}^n\rightarrow \mathbb{P}^n}(P).$
Let $ \Phi \,=\,  \phi^{-1}\circ j (\phi{\mid_{X}})$. So, $ \Phi \,\in\, \textnormal{Aut}_{\mathbb{P}^n}(S_{\mathbb{P}^n}^{r}P),$ and $ \Phi{\mid_{X}}\,=\, id.$ For each $ x\,\in\, \mathbb{P}^n,$ if $ \Phi_{x}$ denotes the restriction of $ \Phi$ to the fiber $ S^{r}P_{x}$ over $x,$ then $ \Phi_{x}\,\in\, \textnormal{Aut}(S^{r}P_{x})$, and restriction of $ \Phi_{x}$ to the diagonal embedding of $P_{x}$ is identity. By  {\cite[Theorem 3.1.1]{Sh}}, we see that $ \Phi_{x}\,=\, id $ for all $x\,\in\, \mathbb{P}^n.$ So, $ \Phi \,=\, id,$ hence, $ \phi\,=\, j(\phi{\mid_{X}})\,\in\, \textnormal{im}(j).$ 
\end{proof}
\section{Acknowledgements}\vskip 0.5mm We thank Prof. János Kollár and Prof. Nagaraj D.S. for insightful discussions.
\printbibliography
\end{document}